\documentclass[conference]{IEEEtran}
\IEEEoverridecommandlockouts

\usepackage{amsmath,amssymb,amsthm}
\usepackage{mathrsfs}
\usepackage{bm}
\usepackage{comment}
\usepackage{graphicx}
\usepackage{subcaption}

\newtheorem{lemma}{Lemma}
\newtheorem{proposition}{Proposition}
\newtheorem{corollary}{Corollary}
\theoremstyle{definition}
\newtheorem{definition}{Definition}
\theoremstyle{remark}
\newtheorem{remark}{Remark}

\DeclareMathOperator{\Tr}{Tr}
\DeclareMathOperator{\conv}{conv}
\DeclareMathOperator{\proj}{proj}

\begin{document}

\title{Copositive Characterizations of Convex Hull Pricing}

\renewcommand{\IEEEauthorrefmark}[1]
{\textsuperscript{#1}}
\author{
\IEEEauthorblockN{Madhusudan Ghosh\textsuperscript{1},
Antoine Lesage-Landry\textsuperscript{2},
and Joshua Adam Taylor\textsuperscript{1}}
\IEEEauthorblockA{\textsuperscript{1}Department of Electrical and Computer Engineering,
New Jersey Institute of Technology, Newark, NJ, USA\\
\textsuperscript{2}Department of Electrical Engineering,
Polytechnique Montreal, GERAD, and Mila, Montreal, QC, Canada\\
Email: mg2262@njit.edu, antoine.lesage-landry@polymtl.ca, jat94@njit.edu}

\thanks{This work was supported by the National Science Foundation under Grant No.~2422849, and the Natural Sciences and Engineering Research Council of Canada (NSERC) through the Alliance grant ALLRP~590233-23.}
}

\maketitle

\begin{abstract}
Due to the nonconvex binary constraints of unit commitment (UC), no uniform linear pricing scheme supports the optimal dispatch. Convex hull pricing (CHP) and copositive duality pricing (CDP) both address this problem. CHP derives the price from the subgradient of the value function of the convex hull relaxation of UC. CDP refers to several different pricing mechanisms that can be constructed from the dual multipliers of the completely positive programming reformulation. In this work, we define a centralized convex hull price over the joint feasible set of UC and prove that, under non-degeneracy, it coincides with the marginal copositive duality price. Numerical experiments on the Scarf example validate this equivalence and quantify the pricing gap introduced by the semidefinite restriction.\\
\end{abstract}


%======================================================================
\section{Introduction}\label{sec:intro}
%======================================================================
Wholesale electricity markets schedule generation by solving a unit commitment (UC) problem. The commitment variables are binary, which makes UC a mixed-integer linear program (MILP) with a non-convex value function~\cite{carrion_arroyo_2006}. A uniform set of energy prices therefore cannot, in general, both efficiently clear the market and leave every generator revenue adequate~\cite{liberopoulos_critical_2016}. Operators close this gap with side payments, called uplift. Uplift covers a
generator's cost not recovered by its market revenue.

Convex hull pricing (CHP), introduced by Hogan and Ring~\cite{hogan_minimum-uplift} and formalized by Gribik, Hogan, and Pope~\cite{gribik_market_clearing}, sets a uniform price that minimizes the total uplift payment. The price is a subgradient of the UC value function, evaluated over its convex hull relaxation. The convex hull relaxation is typically taken over the individual generator feasible sets. This admits a decentralized pricing interpretation and leads to somewhat lighter computations~\cite{hua_convex_2017}.

In a separate line of work, Guo, Bodur, and Taylor~\cite{guo_copositive_2025} introduced copositive duality pricing (CDP). The UC problem is lifted into a completely positive program (CPP) using the reformulation of Burer~\cite{burer_copositive_2009}. The pricing mechanisms are constructed from the dual multipliers of the associated copositive program (COP).

\begin{comment}
The two prices are computed by different means. The convex hull price is a subgradient of the UC value function, while the copositive duality price comes from the multipliers of a conic dual of the same problem. In practice, they yield nearly identical prices. Whether, and under what conditions, the two prices should coincide has remained an open question.
\end{comment}

This paper connects CHP and CDP. We define a centralized convex hull price over the joint feasible set of UC, rather than the per-generator convexification used in the literature. We prove that this price and the copositive duality price are subgradients of a common value function, and that they agree wherever that function is differentiable. At commitment transitions, where the value function is nondifferentiable the two prices may differ, though they share a common subgradient. We further show that the copositive duality price decomposes into two components, one linear in demand and one from the lift. Then we give a copositive formulation of the standard decentralized price, together with a more tractable semidefinite programming (SDP) restriction. Numerical experiments on Scarf's example confirm the equivalence and measure the gap left by the restriction.

The remainder of the paper is organized as follows. Section~\ref{sec:formulations} defines the UC problem and its value function. Section~\ref{sec:chp} develops the convex hull price. Section~\ref{sec:cdp} presents the CPP reformulation and its copositive dual. Section~\ref{sec:equiv} establishes the equivalence of the two prices. Section~\ref{sec:sdp} gives the semidefinite restriction, and Section~\ref{sec:cdp-d} the copositive formulation of the decentralized convex hull price. Section~\ref{sec:numerics} reports numerical results.

%======================================================================
\section{Unit Commitment}\label{sec:formulations}
%======================================================================

We index generators by~$g \in \mathcal{G} \subset \mathbb{N}$ and time periods by~$t \in \mathcal{T} = \{1, \ldots, T\}$, where $T \in \mathbb{N}$ is the time horizon. Each generator has an individual feasible set $\mathcal{X}_g \subset \mathbb{R}^{n_g}$ whose variables are the production levels~$p_g \in \mathbb{R}_+^T$, the binary commitment indicators~$z_g \in \{0,1\}^T$, the startup indicators~$u_g$, and any slacks needed to write capacity limits, minimum up-time and down-time requirements, and startup logic~\cite{Taylor_2015}.

Let $x^{\top} = (u^{\top}, z^{\top}, p^{\top}, \ldots)$. We write the
generator feasible sets explicitly as
$$ \mathcal{X}_g = \bigl\{ x \in \mathbb{R}^{n_g}_+ : a^{\top}_{j,g,t} x = b_{j,g,t}\;\forall j, t;\; z_{g,t} \in \{0,1\}\;\forall t \bigr\}, $$
\noindent where $a_{j,g,t}$ and $b_{j,g,t}$ contain the data of generator $g$'s linear operational constraints.~\cite{carrion_arroyo_2006}
% Inequality constraints are cast as equalities through nonnegative slack variables, consistent with the standard form $x \in \mathbb{R}^{n_g}_+$.

Let the vector~$h_t \in \mathbb{R}^n$ be such that $h_t^\top x = \sum_{g \in \mathcal{G}} p_{g,t}$.
We represent the demand balance in period $t \in \mathcal{T}$ as the hyperplane
\begin{equation*}\label{eq:demand-hyperplane}
\mathcal{H}_t := \Bigl\{ x : \sum_{g \in \mathcal{G}} p_{g,t} = d_t \Bigr\},
\end{equation*}
\noindent where $d_t$ is the total grid demand at time $t$.

Let $\mathcal{X}_{\mathcal{G}} := \prod_{g \in \mathcal{G}} \mathcal{X}_g$ denote the joint generator constraints. The feasible set of the UC problem is
\begin{equation*}\label{eq:joint-feasible}
\mathcal{X}_{\mathrm{UC}} := \mathcal{X}_{\mathcal{G}} \cap \Bigl( \bigcap_{t \in \mathcal{T}} \mathcal{H}_t \Bigr).
\end{equation*}

Define the cost function as,
\[
c^\top x = \sum_{g \in \mathcal{G}} \sum_{t \in \mathcal{T}}
\bigl( c^{\text{p}}_g\, p_{g,t} + c^{\text{u}}_g\, u_{g,t} \bigr),
\]
where $c^{\text{p}}_g$ and $c^{\text{u}}_g$ are the marginal production cost and the startup cost of generator~$g$, respectively. UC can be written as the following MILP
\begin{subequations}\label{eq:UC_MILP}
\begin{align}
\textrm{UC}^{\textup{MILP}}: \;\; \min_x \;\; & c^\top x \label{eq:1a} \\
\text{s.t.} \;\; & x \in \mathcal{H}_t, && \forall t \label{eq:1b} \\
& x \in \mathcal{X}_{\mathcal{G}}. \label{eq:1c}
\end{align}
\end{subequations}
Only the demand balance~\eqref{eq:1b} links one generator to another.

\begin{definition}[Value function]\label{def:v}
The value function $v$ assigns to each demand vector the optimal cost of~\eqref{eq:UC_MILP}:
\begin{equation*}\label{eq:v-def}
v(d) := \min \bigl\{ c^\top x : x \in \mathcal{X}_{\mathrm{UC}}(d) \bigr\}.
\end{equation*}
\end{definition}
\noindent Here, $v$ is piecewise linear and may be discontinuous at demand levels where a generator with positive startup cost must be brought online.

%======================================================================
\section{Convex Hull Pricing}\label{sec:chp}
%======================================================================

Let $\conv(\cdot)$ denote the convex hull operator. We denote the optimal value of the centralized convex hull relaxation by
\begin{equation*}\label{eq:vch-equals-vmilp}
v_{\textup{CHP}}(d) := \min \bigl\{ c^\top x : x \in \conv(\mathcal{X}_{\mathrm{UC}}(d)) \bigr\}.
\end{equation*}
As unit commitment is an MILP, $v_{\textup{CHP}}(d) = v(d)$.

\begin{definition}[Centralized convex hull price]\label{def:chp-subgrad}
At demand $d$, the centralized convex hull price is any subgradient of
the value function,
\begin{equation*}\label{eq:chp-subgrad}
\lambda_{\textup{CHP}} \in \partial v(d).
\end{equation*}
\end{definition}
\noindent When $v$ is differentiable at $d$, Definition~\ref{def:chp-subgrad}
reduces to the gradient $\nabla v(d)$,
in which case the price is unique.

We observe that
\begin{equation}\label{eq:inclusion}
\begin{aligned}
\conv(\mathcal{X}_{\mathrm{UC}}) 
&= \conv\Bigl( \mathcal{X}_{\mathcal{G}} \cap \bigcap_{t \in \mathcal{T}} \mathcal{H}_t \Bigr) \\
&\subseteq \prod_{g \in \mathcal{G}} \conv(\mathcal{X}_g) \cap \bigcap_{t \in \mathcal{T}} \mathcal{H}_t.
\end{aligned}
\end{equation}

The price $\lambda_{\textup{CHP}}$ is hard to express in terms of dual variables because it is not clear how $d$ affects the constraints that make up $\conv(\mathcal{X}_{\mathrm{UC}})$. The above inclusion~\eqref{eq:inclusion} motivates the following approximation of convex hull price $\lambda_{\textup{CHP}}$, which is standard in the literature \cite{hogan_minimum-uplift, gribik_market_clearing}.

Consider the LP
\begin{subequations}\label{eq:UC_CHPD}
\begin{align}
\textrm{UC}^{\textup{CHP-D}}: \;\; \min_{x} \;\;
& c^\top x \label{eq:chpd_obj} \\
\text{s.t.} \;\;
& \textstyle\sum_{g} p_{g,t} = d_t, && \forall t \label{eq:chpd_demand} \\
& x \in \textstyle\prod_{g \in \mathcal{G}} \conv(\mathcal{X}_g). \label{eq:chpd_conv}
\end{align}
\end{subequations}
This has the same objective as $\textrm{UC}^{\textup{MILP}}$. The demand, $d_t$, only appears in~\eqref{eq:chpd_demand}.

\begin{definition}[Decentralized convex hull price]\label{def:chpd}
The decentralized convex hull price $\lambda_{\textup{CHP-D}}$  is the optimal dual variable of the demand balance constraint in $\textrm{UC}^{\textup{CHP-D}}$, with each generator feasible set $\mathcal{X}_g$ convexified separately.
\end{definition}

%======================================================================
\section{Copositive Duality}\label{sec:cdp}
%======================================================================

Now we use copositive duality to explicitly state the centralized convex hull price, $\lambda_{\textup{CHP}}$. The completely positive cone $\mathcal{C}_n^*$ is generated by the nonnegative rank-one matrices:
\begin{equation*}\label{eq:cp_cone}
\mathcal{C}_n^* = \Bigl\{ \sum_k z_k z_k^\top : z_k \in \mathbb{R}^n_+ \Bigr\}.
\end{equation*}
Its dual, the copositive cone $\mathcal{C}_n$, contains the symmetric matrices whose
quadratic form is nonnegative on the nonnegative orthant:
\begin{equation*}\label{eq:copositive_cone}
\mathcal{C}_n = \{ M \in \mathbb{S}^n : y^\top M y \ge 0 \;\; \forall\, y \in \mathbb{R}^n_+ \},
\end{equation*}
where $\mathbb{S}^n$ is the cone of $n \times n$ symmetric matrices.

Burer~\cite{burer_copositive_2009} showed that any MILP can be represented as a CPP whose feasible set projects onto the convex hull of the MILP's feasible set. \\

Let $Y = \begin{pmatrix} 1 & x^\top \\ x & X \end{pmatrix} \in \mathbb{S}^{n+1}$ be the lifted matrix variable, with $X \in \mathbb{S}^n$ corresponding to  the outer product $x x^\top$. Following~\cite{guo_copositive_2025}, we write the completely positive
program $\mathrm{UC}^{\textup{CPP}}$ as
\begin{subequations}\label{eq:UC_CPP}
\begin{align}
\textrm{UC}^{\textup{CPP}}: \;\; \min_{Y} \;\;
& c^\top x \label{eq:2a} \\
\text{s.t.} \;\;
& \textstyle\sum_g p_{g,t} = d_t, && \forall t \label{eq:2b} \\
& a^\top_{j,g,t} x = b_{j,g,t}, && \forall j, g, t \label{eq:2c} \\
& \Tr(h_t h_t^\top X) = d_t^2, && \forall t \label{eq:2d} \\
& \Tr(a_{j,g,t} a_{j,g,t}^\top X) = b_{j,g,t}^2, && \forall j, g, t \label{eq:2e} \\
& z_{g,t} = Z_{g,t}, && \forall g, t \label{eq:2f} \\
& Y \in \mathcal{C}^*_{n+1}. \label{eq:2g}
\end{align}
\end{subequations}
\noindent Constraints~\eqref{eq:2d}--\eqref{eq:2e} are the squared counterparts of the linear equalities in~\eqref{eq:2b}--\eqref{eq:2c}, obtained by the reformulation-linearization technique (RLT), and~\eqref{eq:2f} ties each binary variable to its own diagonal entry. We write $\mathcal{B} = \mathcal{G} \times \mathcal{T}$ for the index set of these binary lifts. By~\cite[Thm.~2.6]{burer_copositive_2009}, the lifted program has the same optimal value as~\eqref{eq:UC_MILP} under the key assumption that the linear constraints implicitly bound the binary variables in $[0,1]$. $\textrm{UC}^{\textup{MILP}}$ satisfies this because the capacity and startup logic in $\mathcal{X}_g$ inherently enforce $z_{g,t} \le 1$ without adding any separate inequality constraint.

The Lagrangian of $\textrm{UC}^{\textup{CPP}}$ is
\begin{equation}\label{eq:cpp-lagrangian}
\begin{aligned}
L(x, X, \alpha; d) =\;& c^\top x
+ \sum_t \lambda_t \bigl( d_t - \textstyle\sum_g p_{g,t} \bigr) \\
& + \sum_{j,g,t} \phi_{j,g,t} \bigl( b_{j,g,t} - a^\top_{j,g,t} x \bigr) \\
& + \sum_t \Lambda_t \bigl( d_t^2 - \operatorname{Tr}(h_t h_t^\top X) \bigr) \\
& + \sum_{j,g,t} \Phi_{j,g,t} \bigl( b_{j,g,t}^2 - \operatorname{Tr}(a_{j,g,t} a_{j,g,t}^\top X) \bigr) \\
& + \sum_{k \in \mathcal{B}} \delta_k (Z_k - z_k) - \Tr(\Omega Y),
\end{aligned}
\end{equation}
where $\alpha = (\lambda, \Lambda, \phi, \Phi, \delta, \Omega)$ collects all dual variables: $\lambda_t, \Lambda_t, \phi_{j,g,t}, \Phi_{j,g,t}, \delta_k$ for the linear constraints~\eqref{eq:2b}--\eqref{eq:2f}, and the symmetric matrix $\Omega \in \mathbb{S}^{n+1}$ for the conic constraint~\eqref{eq:2g}. Because $x$ and $X$ are blocks of $Y$, each term of~\eqref{eq:cpp-lagrangian} is either independent of $Y$ or linear in $Y$. The $Y$-independent terms form
\[
Q(\alpha; d) := \sum_t \bigl( d_t \lambda_t + d_t^2 \Lambda_t \bigr)
+ \sum_{j,g,t} \bigl( b_{j,g,t} \phi_{j,g,t} + b_{j,g,t}^2 \Phi_{j,g,t} \bigr).
\]
The remaining terms collect into the trace inner product $\Tr(M(\alpha)Y)$ for some symmetric matrix $M(\alpha) \in \mathbb{S}^{n+1}$. This decomposes the Lagrangian as
\begin{equation}\label{eq:L-decomp}
L(x, X, \alpha; d) = Q(\alpha; d) + \Tr(M(\alpha) Y).
\end{equation}
Following conic duality, the infimum of $L$ over $Y \in \mathcal{C}^*_{n+1}$ equals $Q(\alpha; d)$ if $M(\alpha) \in \mathcal{C}_{n+1}$, and is unbounded below otherwise. The dual problem of $\textrm{UC}^{\textup{CPP}}$ is
\begin{subequations}\label{eq:uc_cop}
\begin{alignat}{2}
\textrm{UC}_{d}^{\textup{COP}}: \;\;
& \max_{\lambda,\Lambda,\phi,\Phi,\delta} \;\; && Q(\alpha; d) \label{eq:uc_cop_obj} \\
& \quad \;\; \text{s.t.} \;\; && M(\alpha) \in \mathcal{C}_{n+1}. \label{eq:uc_cop_cone}
\end{alignat}
\end{subequations}
\noindent According to Cifuentes et al.~\cite{cifuentes_sensitivity_2024}, a bounded feasible region guarantees a strictly feasible point of the copositive dual, so by Slater's condition $\textrm{UC}^{\textup{CPP}}$ and $\textrm{UC}_{d}^{\textup{COP}}$
have strong duality. We denote the optimal value of the copositive dual by $v_{\textup{CDP}}(d)$. By Burer's
theorem~\cite[Thm.~2.6]{burer_copositive_2009}, this implies
\begin{equation*}\label{eq:strong-duality}
v_{\textup{CDP}}(d) = v(d).
\end{equation*}
Specializing~\cite[Cor.~2.4]{burer_copositive_2009} to UC we have,
\begin{equation*}\label{eq:burer-projection}
\proj_x\!\bigl( \mathrm{Feas}(\textrm{UC}^{\textup{CPP}}) \bigr)
= \conv(\mathcal{X}_{\mathrm{UC}}),
\end{equation*}
where $\proj_x$ denotes projection onto the $x$ coordinates and $\mathrm{Feas}(\cdot)$ denotes the feasible set of an optimization problem.

%======================================================================
\section{Centralized CHP}\label{sec:equiv}
%======================================================================
We now characterize the exact algebraic relationship between CHP and CDP.

\begin{lemma}\label{lem:subdiff}
The CHP and CDP value functions equal the UC value function, for all $d$,
\begin{equation*}
v_{\textup{CHP}}(d) = v_{\textup{CDP}}(d) = v(d).
\end{equation*}
In particular, at each $d$ there is a common subgradient
$g \in \partial v_{\textup{CHP}}(d) \cap \partial v_{\textup{CDP}}(d) \cap \partial v(d)$.
\end{lemma}
\begin{proof}
The proof is straightforward as it follows from basic results on MILP and strong duality of CPP and COP.
\end{proof}

By Definition~\ref{def:chp-subgrad}, the convex hull price is a subgradient of~$v$. By Lemma~\ref{lem:subdiff}, $v$ is the value of the copositive dual, so the price can be obtained from that dual. Let $\mathcal{S}^*(d)$ denote the set of saddle-point triplets $(x^*, X^*, \alpha^*)$ of the Lagrangian~\eqref{eq:cpp-lagrangian} at demand~$d$. By strong duality,
\begin{equation*}\label{eq:saddle-equality}
    v(d) = L(x^*, X^*, \alpha^*; d), \quad \text{for all } (x^*, X^*, \alpha^*) \in \mathcal{S}^*(d).
\end{equation*}

\noindent Consequently, the exact centralized convex hull price is characterized by the
subdifferential of the Lagrangian evaluated over this saddle-point set. For each
period $t$,
\begin{equation*}\label{eq:chp-characterization}
    \lambda_{\textup{CHP},t}(d) = \Bigl\{ \tfrac{\partial L}{\partial d_t}(x^*, X^*, \alpha^*; d) : (x^*, X^*, \alpha^*) \in \mathcal{S}^*(d) \Bigr\}.
\end{equation*}

\begin{proposition}\label{prop:energy-price}
For any $(x^*, X^*, \alpha^*) \in \mathcal{S}^*(d)$,
\begin{equation}\label{eq:energy-price}
    \frac{\partial L}{\partial d_t}(x^*, X^*, \alpha^*; d) = \lambda^*_t + 2d_t\Lambda^*_t.
\end{equation}
Hence, if $v$ is differentiable at $d$, then $\frac{\partial v(d)}{\partial d_t} = \lambda^*_t + 2d_t\Lambda^*_t$.
\end{proposition}

\begin{proof}
Demand $d$ enters the Lagrangian $L$  only through the terms $d_t\lambda_t$ and $d_t^2\Lambda_t$. Differentiating the Lagrangian at a fixed saddle-point triplet $(x^*, X^*, \alpha^*)$ yields~\eqref{eq:energy-price}.
\end{proof} 
The copositive duality price is thus the sum of the balance dual $\lambda^*$ and a lifted term $2 d_t \Lambda^*$ contributed by the RLT load constraint~\eqref{eq:2d}. 

\begin{remark}\label{rem:sdp-parallel}
The same price mechanism is obtained for the semidefinite relaxation by Guo et al.~\cite{guo_pricing_2026}. 
\end{remark}

\begin{remark} \label{rem:marginal-volumetric}
The price in~\eqref{eq:energy-price} is the marginal value $\partial v / \partial d_t$. It differs from the volumetric per-MW form $\lambda^*_t + d_t \Lambda^*_t$ used in the revenue-adequate scheme of Guo, Bodur, and Taylor~\cite{guo_copositive_2025}, where a uniform price recovers total demand-related cost rather than measuring marginal value. The second term's factor of 2 distinguishes the two forms and is what allows the equivalence with convex hull pricing in Corollary~\ref{cor:equiv}.
\end{remark}

\begin{corollary}\label{cor:equiv}
If $v$ is differentiable at $d$, i.e, $\partial v(d)$ is a
singleton, then the centralized convex hull price and the copositive duality price agree:
\begin{equation}\label{eq:price-equal}
\lambda_{\textup{CHP},t}
= \lambda^*_t + 2 d_t \Lambda^*_t.
\end{equation}
\end{corollary}

\begin{proof}
Lemma~\ref{lem:subdiff} and Proposition~\ref{prop:energy-price} place both prices
in $\partial v(d)$. Differentiability collapses $\partial v(d)$ to the unique
gradient $\nabla v(d)$, and the two prices are identical.
\end{proof}

\begin{remark}\label{rem:degen}
Because $v$ is piecewise linear, it is differentiable except on a measure-zero set. These are the commitment transitions, where two or more integer dispatches share the same cost. At such points the left and right derivatives of $v$ differ, both pricing schemes return one-sided values, and the two may not coincide.
\end{remark}

%======================================================================
\section{Semidefinite Approximation}\label{sec:sdp}
%======================================================================

Murty and Kabadi~\cite{murty_kabadi_1987} showed that testing copositivity is co-NP-complete, so optimizing over $\mathcal{C}_{n+1}$ is intractable. De Klerk and Pasechnik~\cite{deklerk_pasechnik_2002} approximate it from within by the tractable cone $\mathbb{S}^{n+1}_+ + \mathcal{N}_{n+1} \subseteq \mathcal{C}_{n+1}$, where $\mathbb{S}^{n+1}_+ \subset \mathbb{S}^{n+1}$ is the positive semidefinite cone and $\mathcal{N}_{n+1}$ is the cone of entrywise nonnegative symmetric matrices. Replacing $\mathcal{C}_{n+1}$ in~\eqref{eq:uc_cop_cone} with this cone gives the semidefinite restriction
\begin{subequations}\label{eq:uc_sdp}
\begin{align}
\textrm{UC}_{d}^{\textup{SDP}}: \;\;
\max_{\alpha, S, N} \;\; & Q(\alpha; d) \label{eq:uc_sdp_obj} \\
\text{s.t.} \;\; & M(\alpha) = S + N, \label{eq:uc_sdp_cone_sum} \\
& S \succeq 0, \label{eq:uc_sdp_cone_psd} \\
& N \ge 0. \label{eq:uc_sdp_cone_nn}
\end{align}
\end{subequations}
Let $v_{\textup{SDP}}(d)$ denote the optimal value of~\eqref{eq:uc_sdp}. The
restricted cone is smaller, which shrinks the feasible region of the
maximization, which gives 
\begin{equation*}\label{eq:sdp-bound}
v_{\textup{SDP}}(d) \le v(d).
\end{equation*}
Let $(\lambda^*_{\textup{SDP},t}, \Lambda^*_{\textup{SDP},t})$ be the optimal
demand multipliers of~\eqref{eq:uc_sdp}. We define the SDP-restricted CDP price
at demand $d$ as $\lambda^*_{\textup{SDP},t} + 2 d_t \Lambda^*_{\textup{SDP},t}$.
This follows the derivation of~\eqref{eq:energy-price}, with the optimal
multipliers of the semidefinite restriction. It coincides with
$\lambda_{\textup{CHP},t}$ when the gap $v(d) - v_{\textup{SDP}}(d)$ vanishes.

%======================================================================
\section{Decentralized CHP}\label{sec:cdp-d}
%======================================================================

We now re-express the decentralized convex hull price $\lambda_{\text{CHP-D}}$
through copositive duality. Burer's theorem~\cite{burer_copositive_2009} applies to each generator
individually. For each $g$, we have
\begin{equation}\label{eq:burer-per-gen}
\proj_{x_g}\!\bigl( \mathrm{Feas}(\mathcal{X}_g^{\text{CPP}}) \bigr) = \conv(\mathcal{X}_g),
\end{equation}
where $\mathcal{X}_g^{\textup{CPP}}$ lifts generator $g$'s variables $x_g$ to
$Y_g = \bigl(\begin{smallmatrix} 1 & x_g^\top \\ x_g & X_g \end{smallmatrix}\bigr)
\in \mathbb{S}^{n_g+1}$, with $X_g \in \mathbb{S}^{n_g}$ the lift of
$x_g x_g^\top$. We write $\{Y_g\}$ for the collection of these matrices over all
$g \in \mathcal{G}$. This program carries only generator $g$'s own constraints,
the per-generator analogues of~\eqref{eq:2c}, \eqref{eq:2e}, and~\eqref{eq:2f},
together with the cone constraint $Y_g \in \mathcal{C}^*_{n_g+1}$.

Substituting~\eqref{eq:burer-per-gen} into~\eqref{eq:chpd_conv} for all $g \in \mathcal{G}$ gives the decentralized completely positive program
\begin{subequations}\label{eq:UC_CPPD}
\begin{align}
\textrm{UC}^{\textup{CPP-D}}: \;\; \min_{\{Y_g\}} \;\;
& c^\top x \label{eq:cppd_obj} \\
\text{s.t.} \;\;
& \textstyle\sum_g p_{g,t} = d_t, && \forall t \label{eq:cppd_bal} \\
& a^\top_{j,g,t} x = b_{j,g,t}, && \forall j, g, t \label{eq:cppd_op} \\
& \Tr(a_{j,g,t} a_{j,g,t}^\top X_g) = b_{j,g,t}^2, && \forall j, g, t \label{eq:cppd_rlt} \\
& z_{g,t} = Z_{g,t}, && \forall g, t \label{eq:cppd_bin} \\
& Y_g \in \mathcal{C}^*_{n_g+1}, && \forall g. \label{eq:cppd_cone}
\end{align}
\end{subequations}
In place of the single cone constraint $Y \in \mathcal{C}^*_{n+1}$ of~\eqref{eq:UC_CPP}, each generator now carries its own $Y_g \in \mathcal{C}^*_{n_g+1}$. Because the~$Y_g$ share no entries across generators, the demand-RLT~\eqref{eq:2d} cannot be formed, and no multiplier $\Lambda$ appears in the dual.

We dualize~\eqref{eq:UC_CPPD} as in Section~\ref{sec:cdp}. The Lagrangian is linear in each $Y_g$,
so it decomposes as
\begin{equation}\label{eq:L-decomp-D}
L_{\text{D}} = Q_{\text{D}}(\lambda_{\text{D}}, \{\alpha_g\}; d) + \sum_g \Tr(M_g(\lambda_{\text{D}}, \alpha_g)\, Y_g),
\end{equation}
where $\lambda_{\text{D}}$ is the demand-balance multiplier and
$\alpha_g = (\phi_g, \Phi_g, \delta_g, \Omega_g)$ gathers the remaining duals of
generator $g$. The first term of~\eqref{eq:L-decomp-D} defined as,
\[
Q_{\text{D}}(\lambda_{\text{D}}, \{\alpha_g\}; d) = \sum_t d_t \lambda_{{\text{D}},t}
+ \sum_{j,g,t}\bigl(b_{j,g,t}\phi_{j,g,t} + b_{j,g,t}^2 \Phi_{j,g,t}\bigr),
\]
is $Y_g$-independent term, similarly to $Q$  in~\eqref{eq:L-decomp}, now without the $\sum_t d_t^2 \Lambda_t$ term. Each $M_g$ is the coefficient of $Y_g$. The dual is
\begin{subequations}\label{eq:UC_COPD}
\begin{align}
\textrm{UC}_{d}^{\textup{COP-D}}: \;\; \max_{\lambda_{\text{D}},\, \{\alpha_g\}} \;\;
& Q_{\text{D}}(\lambda_{\text{D}}, \{\alpha_g\}; d) \label{eq:copd_obj} \\
\text{s.t.} \;\; & M_g(\lambda_{\text{D}}, \alpha_g) \in \mathcal{C}_{n_g+1},
  && \forall g. \label{eq:copd_cone}
\end{align}
\end{subequations}
The cone constraints~\eqref{eq:copd_cone} separate across generators, coupling only through the shared multiplier $\lambda_{\text{D}}$.

Demand enters~\eqref{eq:UC_COPD} only through $\sum_t d_t \lambda_{{\text{D}},t}$. By the differentiation argument of
Proposition~\ref{prop:energy-price},
\begin{equation*}\label{eq:chpd-price}
\lambda_{\textup{CHP-D},t} = \lambda^*_{{\text{D}},t}.
\end{equation*}
The centralized price~\eqref{eq:price-equal} carried a second term $2 d_t \Lambda^*_t$. That term is absent here, because the lifted demand balance constraint is absent.

Restricting each cone $\mathcal{C}_{n_g+1}$ in~\eqref{eq:copd_cone}, as in Section~\ref{sec:sdp}, gives the decentralized semidefinite program:
\begin{subequations}\label{eq:UC_SDPD}
\begin{alignat}{2}
\textrm{UC}_{d}^{\textup{SDP-D}}: 
& \max_{\lambda_{\text{D}}, \{\alpha_g\}, S_g, N_g}  Q_{\text{D}}(\lambda_{\text{D}}, \{\alpha_g\}; d) \label{eq:sdpd_obj} \\
& \quad \quad \;\; \text{s.t.} \quad \;\;\; M_g(\lambda_{\text{D}}, \alpha_g) = S_g + N_g, \;\; \forall g, \label{eq:sdpd_cone_sum} \\
& \quad \quad \quad \quad \quad \;\; S_g \succeq 0, \quad \forall g, \label{eq:sdpd_cone_psd} \\
& \quad \quad \quad \quad \quad \;\; N_g \ge 0, \quad \forall g. \label{eq:sdpd_cone_nn}
\end{alignat}
\end{subequations}
Let $\lambda^*_{\textup{SDP-D},t}$ be the optimal demand multiplier of~\eqref{eq:UC_SDPD}. The SDP-restricted decentralized CDP price at demand $d$ is $\lambda^*_{\textup{SDP-D},t}$. 

%======================================================================
\section{Numerical Examples}\label{sec:numerics}
%======================================================================

We compare four prices: the decentralized convex hull price $\lambda_{\text{CHP-D}}$, the decentralized copositive duality price $\lambda^*_{\text{D},t}$, the SDP-restricted decentralized CDP price $\lambda^*_{\text{SDP-D},t}$, and the centralized copositive duality price $\lambda^*_t + 2d_t\Lambda^*_t$. By the equivalence of Section~\ref{sec:equiv}, the exact copositive price coincides with the centralized convex hull price $\lambda_{\textup{CHP}}$ where the value function is differentiable. All experiments were implemented in Python using CVXPY, with Gurobi, MOSEK, and
Clarabel as solvers.

We use the modified Scarf's example of Hogan and Ring~\cite{hogan_minimum-uplift}, a standard benchmark for nonconvex-market pricing. It has three generator types, with parameters listed in Table~\ref{tab:scarf}.

\begin{table}[h]
\centering
\caption{Generator types in the modified Scarf example.}
\label{tab:scarf}
\begin{tabular}{lcccc}
\hline
Type & Count & $c^{\text{p}}$ (\$/MWh) & $c^{\text{u}}$ (\$) & Capacity (MW) \\
\hline
Smokestack        & 6 & 3 & 53 & 16 \\
High tech.   & 5 & 2 & 30 & 7 \\
Medium tech. & 5 & 7 & 0  & 6 \\
\hline
\end{tabular}
\end{table}
We vary the demand from $10$ to $160$~MW in increments of $5$~MW and compare the four prices and the make-whole uplift each leaves.  

Let $x_g^*$ be the dispatch assigned to
generator~$g$, with $p_{g,t}^*$ its production in period~$t$. Its cost is $C_g = c_g^\top x_g^*$ and its revenue at the price~$\lambda_t$ is $R_g = \sum_t \lambda_t p_{g,t}^*$. The uplift of generator~$g$ is the make-whole payment it receives when its revenue falls short of its
cost~\cite{guo_uplift},
\begin{equation}\label{eq:uplift}
U_g(\lambda) = \max\bigl( 0,\; C_g - R_g \bigr),
\end{equation}
and the total uplift is $U(\lambda) = \sum_g U_g(\lambda)$.
\begin{figure}[t]
  \centering
  \begin{subfigure}{\columnwidth}
    \centering
    \includegraphics[width=0.95\columnwidth]{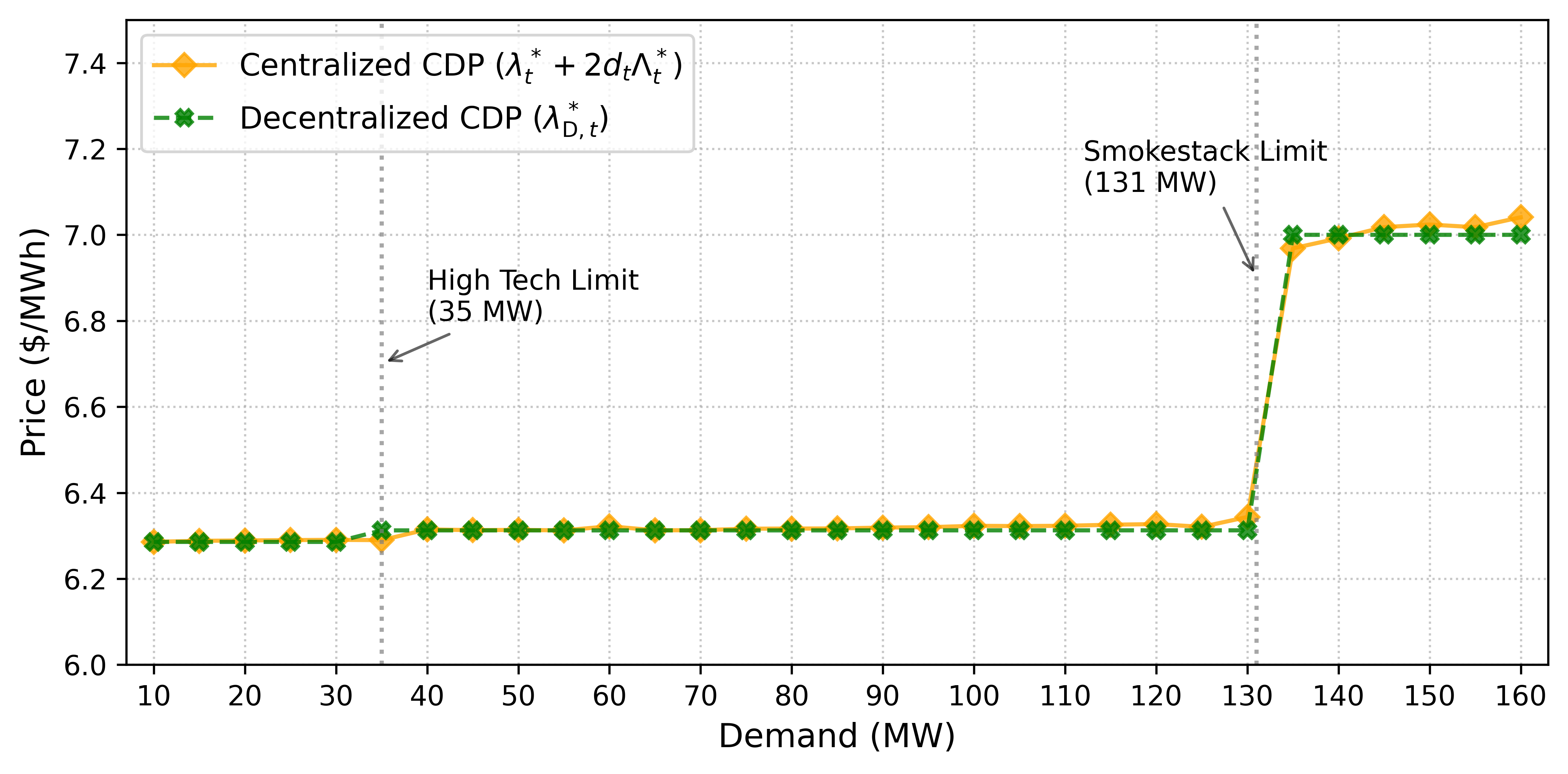}
    \caption{Centralized and decentralized CDP prices.}
    \label{fig:price-a}
  \end{subfigure}
  \\[0.6em]
  \begin{subfigure}{\columnwidth}
    \centering
    \includegraphics[width=0.95\columnwidth]{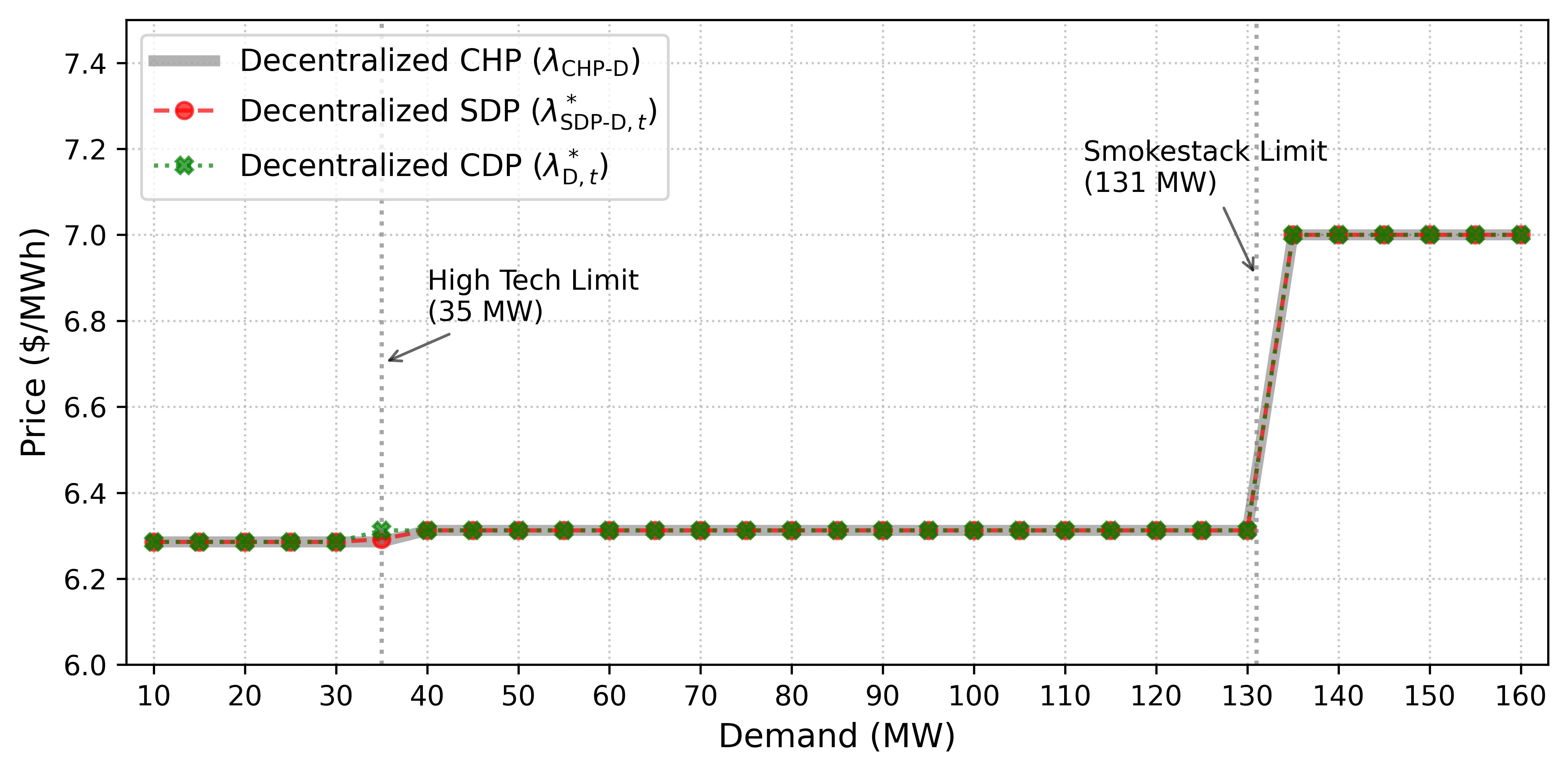}
    \caption{All three decentralized prices.}
    \label{fig:price-b}
  \end{subfigure}
  \caption{Price comparison across demand levels.}
  \label{fig:price}
\end{figure}

\begin{figure}[t]
  \centering
  \begin{subfigure}{\columnwidth}
    \centering
    \includegraphics[width=0.95\columnwidth]{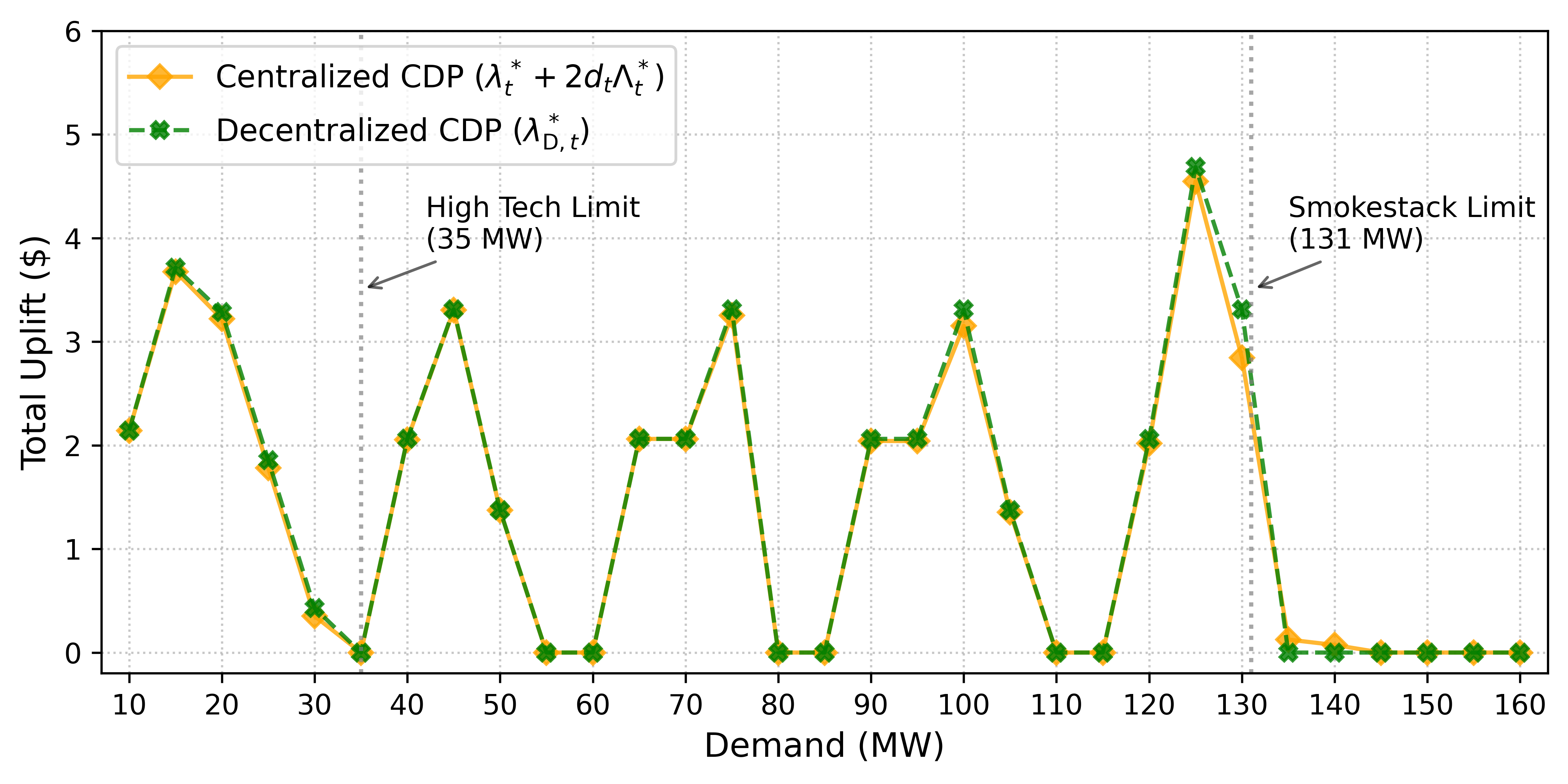}
    \caption{Uplift for centralized and decentralized CDP prices.}
    \label{fig:uplift-a}
  \end{subfigure}
  \\[0.6em]
  \begin{subfigure}{\columnwidth}
    \centering
    \includegraphics[width=0.95\columnwidth]{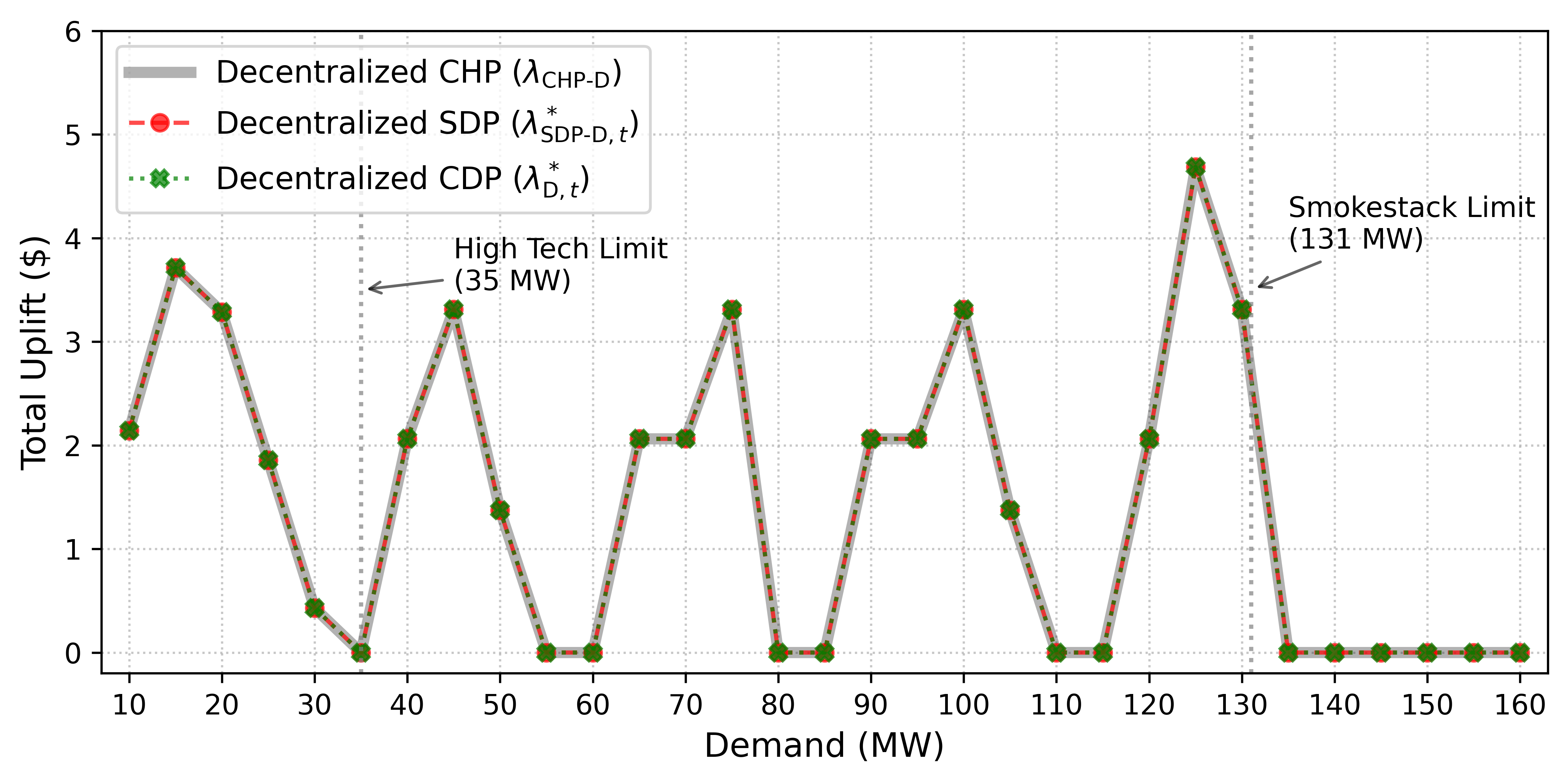}
    \caption{Uplift for all three decentralized prices.}
    \label{fig:uplift-b}
  \end{subfigure}
  \caption{Make-whole uplift across demand levels.}
  \label{fig:uplift}
\end{figure}

Figure~\ref{fig:price}(a) compares the two copositive duality prices. At higher
demand levels the centralized price is slightly higher than the decentralized
price. Figure~\ref{fig:price}(b) shows that the three decentralized prices
coincide across all demand levels, except for the transition at $35$~MW, where
the subgradient is not a singleton.

Figure~\ref{fig:uplift}(a) compares uplifts under centralized and decentralized copositive duality pricing. The decentralized price results in more uplift than the centralized price. However, this does not contradict the uplift-minimizing property of decentralized convex hull pricing. The decentralized convex hull price minimizes uplift only over uniform linear prices. The centralized price is nonlinear, because $\Lambda^*_t$ is the multiplier of a squared constraint. Under the lost-opportunity-cost definition of uplift~\cite{gribik_market_clearing,guo_uplift}, the standard decentralized convex hull price results in less uplift than the centralized price at every demand level, as the theory predicts (not shown here). Figure~\ref{fig:uplift}(b) shows that the uplifts of the three decentralized prices coincide, as their prices do.

%======================================================================
\section{Conclusion}\label{sec:conclusion}
%======================================================================

We have characterized the relationship between convex hull pricing and copositive duality pricing for unit commitment. Defining the centralized convex hull price over the full feasible set, we showed that it and the copositive duality price are subgradients of a common value function and coincide wherever that function is differentiable. The copositive price includes a lifted term beyond the demand price, which is not present in standard decentralized convex hull pricing. We gave the copositive formulation of that decentralized price, and a more tractable semidefinite restriction. Numerical experiments on the Scarf example confirm the equivalence and compare the prices and uplift of each scheme.

% --- Bibliography placeholder ---
\bibliographystyle{ieeetr}
\bibliography{Cite}

\end{document}